\documentclass[reqno,a4paper]{amsart}%%{ctexart}
\usepackage{amsmath,amstext,amssymb,amsfonts,amscd,amsthm}
\usepackage{geometry}
\usepackage{hyperref}
\hypersetup{pdfpagemode=FullScreen,colorlinks=true,linkcolor=blue,citecolor=blue,urlcolor=orange}
\usepackage{mathdots,mathrsfs,enumerate}
\usepackage{subfigure}
\usepackage{extarrows}
\usepackage{mathrsfs}  %% 数学RSFS 书写字体
\usepackage{dsfont}  %% 数学双击字体；像整数集，实数集符号，一般是需要的。
\usepackage{graphicx,color}
\usepackage{pgffor} %%可以使用\foreach,\breakforeach 的包 %%在 TikZ 单是\foreach不需要, 在一般的 LaTeX 需要
\usepackage{ifthen} %%可以使用\ifthenelse的包, 还能使用\whiledo
\usepackage{tikz}
\usetikzlibrary{3d,calc,patterns,graphs,arrows,intersections,through}
\usetikzlibrary{shapes,shapes.geometric}%%2024年2月8 日- 画多边形
\usetikzlibrary{decorations.text}%%将文字绑定到路径上
\usepackage{tikz-cd,array,diagbox}
\usepackage{framed}

\usepackage{pgfplots} %%\begin{axis}... \addplot3

\usepackage{ulem}%%给文字添加下划线、波浪线等样式
\usepackage{mathtools}%%右花括号\begin{rcases}

\usepackage{verbatim}

\numberwithin{equation}{section}
\newenvironment{proof*}{\noindent{\heiti 证明}}{\hfill\qed}%

\newtheorem{lemma}{Lemma}[section]
\newtheorem{theorem}{Theorem}[section]

\newtheorem{proposition}{Proposition}[section]

\newtheorem{remark}{Remark}[section]

\newcommand{\A}{\forall}

\newcommand{\p}{\partial}

\renewcommand{\t}{\triangle}

\newcommand{\abs}[1]{\left\vert#1\right\vert}

\newcommand{\set}[1]{\left\{#1\right\}}

\newcommand{\frb}[1]{\left(#1\right)}

\newcommand{\ol}{\overline}

\newcommand{\wt}{\widetilde}

\newcommand{\Z}{\mathbb Z}

\newcommand{\R}{\mathbb R}

\renewcommand{\a}{\alpha}
\renewcommand{\b}{\beta}
\newcommand{\g}{\gamma}
\newcommand{\G}{\Gamma}
\renewcommand{\d}{\delta}

\newcommand{\ve}{\varepsilon}

\renewcommand{\l}{\lambda}
\renewcommand{\L}{\Lambda}

\newcommand{\s}{\sigma}

\newcommand{\op}[1]{\operatorname{#1}} %%\operatorname
\allowdisplaybreaks[3]

\begin{document}

\title[A concavity inequality and interior $C^2$ estimate]
{A concavity inequality and interior $C^2$ estimate for Hessian quotient equations}

\author{Zhisu Li and Ke Wu}

\address{School of Mathematics, Northwest University}

\date{\today}

\maketitle
\tableofcontents

\begin{abstract}
We establish a concavity inequality
for the Hessian quotient operators $\dfrac{\s_k}{\s_l}$
in the cases $k-l\in\{1,2\}$,
and then derive the corresponding Jacobi inequality.
Combining this with the framework developed by Lu and Tsai \cite{LT26},
we obtain an interior Hessian estimate for convex solutions of
$\dfrac{\s_k(D^2u)}{\s_l(D^2u)}=f$.
%%
%%Together with the known singular examples for $k-l\geq3$,
%%this completes the interior $C^2$ regularity classification for convex solutions of Hessian quotient equations.
%%
As an application,
we prove that any entire convex solution in $\R^n$ with quadratic growth must be a quadratic polynomial.
\end{abstract}



\section{Introduction}

Interior Hessian estimates are a fundamental problem in the theory of fully nonlinear elliptic equations.
The study of such estimates can be traced back to Heinz \cite{H59},
who established an interior $C^2$ estimate for the two-dimensional Monge--Amp\`ere equation,
with later proofs given in \cite{CHO16,Liu21}.
However, Pogorelov \cite{P78} constructed singular solutions in
dimensions $n\geq3$, showing that such an estimate fails without further structure.
Later, Urbas \cite{U90} showed that an interior $C^2$ estimate also fails
for the Hessian equation $\s_k(D^2u)=f$ for $k\geq3$.
This leaves the quadratic Hessian equation $\s_2(D^2u)=f$ as the remaining case.

For $\s_2(D^2u)=1$ in dimension three,
Warren and Yuan \cite{WY09} obtained an interior Hessian estimate,
which was extended by Qiu \cite{Q24} to general positive right-hand sides.
In general dimensions,
interior Hessian estimates have been obtained under various structural assumptions.
McGonagle, Song and Yuan \cite{MSY19} and Shankar and Yuan \cite{SY20}
treated almost convex and semiconvex solutions, respectively, for $\s_2(D^2u)=1$,
while Mooney \cite{M21} considered convex viscosity solutions.
Guan and Qiu \cite{GQ19} obtained the estimate for $\s_2(D^2u)=f$ under the condition $\s_3(D^2u)\geq-A$.
In 2025, Shankar and Yuan \cite{SY25} proved the interior Hessian
estimate for $\s_2(D^2u)=1$ in dimension four
and obtained corresponding higher-dimensional estimates
under the dynamic semiconvexity condition $\l_{\min}(D^2u)\geq-c(n)\t u$,
while Fan \cite{Fan26} extended the dimension-four result to variable right-hand sides.
Most recently, Chen, Jian, Tu and Zhou \cite{CJTZ26}
established interior $C^2$ regularity for convex viscosity solutions of
$\s_2(D^2u)=f(x)$ in all dimensions, assuming only $f\in C^{0,1}$ and $\inf f>0$.

A closely related class is given by the Hessian quotient equations
\[
F(D^2u):=\frac{\s_k(D^2u)}{\s_l(D^2u)}=f(x,u)
\quad\text{for}\ 1\leq l<k\leq n.
\]
For $\s_3/\s_1$ in dimensions three and four,
interior Hessian estimates were obtained by exploiting the special Lagrangian structure;
see \cite{CWY09,WY14,Z24}.
Lu \cite{Lu23} later obtained a Jacobi-inequality proof in dimension three.
Using a concavity inequality from \cite{GS26},
Lu \cite{Lu25} further combined the Jacobi inequality with the Legendre transform
to obtain interior estimates
for the positive quotients $\s_n/\s_{n-1}$ and $\s_n/\s_{n-2}$ in general dimensions.
In the same work,
he also constructed singular solutions for the general quotient equation when $l\leq k-3$.
Jiao and Sui \cite{JS26} treated $\s_2/\s_1$ for $2$-convex solutions in
dimension three and for semiconvex solutions in higher dimensions.
More recently, Mei and Yan \cite{MY26} established estimates
for semiconvex solutions of $\s_3/\s_l=1$ with $l=1,2$ in arbitrary dimensions,
whereas Fung \cite{F26} developed a pointwise doubling method for adjacent
quotients with gradient-dependent right-hand sides.
For general adjacent quotients, however,
Fung's argument still relies on a Lu--Tsai-type structural concavity condition.

In view of the preceding results,
the unresolved cases for convex solutions of the general quotient equation
are $l=k-1$ and $l=k-2$.
Lu and Tsai \cite{LT26} obtained the interior Hessian estimate
in these cases under \cite[Assumption~1.1]{LT26}.
Their proof first uses this assumption to derive the Jacobi inequality
for $b=\log\l_{\max}(D^2u)$, see \cite[Lemma~2.5]{LT26},
and then converts it into the Hessian estimate
through the Legendre transform, the mean value inequality,
and integration by parts in \cite[Sections~3 and~4]{LT26}.
Thus, the remaining issue in their framework is the algebraic concavity inequality
required in the largest eigenvalue calculation.

The assumption in \cite{LT26} is imposed for every $\xi\in\R^n$,
whereas the Jacobi calculation only requires a restricted class of directions.
Indeed, after diagonalizing $D^2u$ at a point
where the largest eigenvalue has multiplicity $m$,
the multiple eigenvalue derivative formula gives
\[
u_{111}=u_{221}=\cdots=u_{mm1}.
\]
Hence the vector $\xi=(u_{ii1})$ arising in the Jacobi calculation
satisfies $\xi_1=\xi_2=\cdots=\xi_m$.
It is therefore sufficient to establish the required concavity inequality
for this restricted class of directions.
When $k-l=1$, $F$ itself is the adjacent quotient $\s_k/\s_{k-1}$.
When $k-l=2$, the key structural observation is the one-dimensional lifting
\[
\L=\frb{\l,\sqrt{F(\l)}}\in\R^{n+1},
\quad
\frac{\s_k(\L)}{\s_{k-1}(\L)}=\sqrt{F(\l)},
\]
which reduces the problem to an adjacent quotient in one higher dimension.
This leads to the following theorem.

\begin{theorem}\label{thm.quotient-concavity}
Let $1\leq l<k\leq n$, $k-l\in\set{1,2}$ and
$0<\d<1/2$.
There exist constants
$A=A(n,k,l,\d)\geq2$ and $K=K(n,k,l,\d)\geq1$
such that, whenever
\[
\l_1=\l_2=\cdots=\l_m=:\l_{\max}>
\l_{m+1}\geq\cdots\geq\l_n\geq0,
\quad
\s_k(\l)>0
\]
for some $m\in\Z_{[1,n]}$, and
\[
\l_{\max}\geq AF(\l)^{1/(k-l)},
\]
we have
\begin{equation}\label{eqn.quotient-concavity}
-\sum_{i,j=1}^nF_{ij}\xi_i\xi_j
+\frac{K}{F}\frb{\sum_{i=1}^nF_i\xi_i}^2
+\sum_{i>m}\frac{2F_i\xi_i^2}{\l_{\max}-\l_i}
\geq
(1+\d)\frac{F_1\xi_1^2}{\l_{\max}}
\end{equation}
for every $\xi\in\R^n$ satisfying
$\xi_1=\xi_2=\cdots=\xi_m$,
where $F_i$ and $F_{ij}$ denote derivatives
with respect to the eigenvalue variables.
When $m=n$,
the terms involving $\l_{m+1}$ and the sums over $i>m$
are understood to be omitted.
\end{theorem}

\begin{remark}
For $k=n$, Guan and Sroka \cite{GS26} established a stronger concavity inequality
for positive Hessian quotient operators, which holds for arbitrary $\xi\in\R^n$.
For $l=0$, so that $F=\s_k$ is the $k$-Hessian operator,
Zhang \cite{Z25} obtained the corresponding inequality under a semiconvexity assumption
when $\l_{\max}$ is sufficiently large,
again for arbitrary $\xi\in\R^n$.
\end{remark}

\begin{remark}
After the completion of this work,
we became aware of a recent preprint by Tsai \cite{T26},
who independently proved \cite[Assumption~1.1]{LT26} for $\s_k/\s_{k-1}$
using a change of basis for symmetric polynomials.
\end{remark}

Theorem \ref{thm.quotient-concavity}
implies the Jacobi inequality required in the framework of \cite{LT26}.
Following the arguments in Sections~3 and~4 of \cite{LT26},
we obtain the following interior Hessian estimate.

\begin{theorem}\label{thm.main}
Let $n\geq2$, $1\leq l<k\leq n$ with
$k-l\in\{1,2\}$.
Let $f\in C^{1,1}(B_{10}\times\R)$ be a positive function.
Suppose that $u\in C^4(B_{10})$ is a convex solution of
\[
F(D^2u)
=\frac{\s_k(D^2u)}{\s_l(D^2u)}
=f(x,u)
\quad\mathrm{in}\ B_{10}.
\]
Then we have
\[
|D^2u(0)|\leq C,
\]
where $C$ is a positive constant depending only on $n$, $k$, $l$,
$\|u\|_{C^{0,1}\frb{\ol{B_9}}}$,
$\min_{\ol{B_9}\times[-M,M]}f$
and $\|f\|_{C^{1,1}\frb{\ol{B_9}\times[-M,M]}}$,
and $M$ is a large constant satisfying $\|u\|_{C^{0,1}\frb{\ol{B_9}}}\leq M$.
\end{theorem}

Together with the known singular examples for $l\leq k-3$,
Theorem \ref{thm.main} completes the interior $C^2$ regularity classification
for convex solutions of Hessian quotient equations.

Interior Hessian estimates are closely related to Liouville theorems
for entire solutions with quadratic growth.
For Hessian quotient equations,
Bao, Chen, Guan and Ji \cite{BCGJ03}
proved such a theorem for entire convex solutions of the top Hessian quotients.
More recently, Lu and Sroka \cite{LS26}
established a Liouville theorem for semiconvex entire solutions of $\s_2/\s_1=1$.
As an application of Theorem \ref{thm.main},
we extend this rigidity phenomenon to the Hessian quotient equations
with $k-l\in\set{1,2}$ considered in this paper.

\begin{theorem}\label{thm.liouville}
Let $n\geq2$, $1\leq l<k\leq n$ with $k-l\in\{1,2\}$.
Suppose that $u\in C^4(\R^n)$ is a convex solution of
\[
\frac{\s_k(D^2u)}{\s_l(D^2u)}=1
\quad\mathrm{in}\ \R^n.
\]
Assume that there exist constants $a>0$, $c\geq0$, and $R>0$ such that
\[
u(x)\geq a|x|^2-c
\quad\text{for all}\ |x|\geq R.
\]
Then $u$ is a quadratic polynomial.
\end{theorem}

The paper is organized as follows.
In Section \ref{sec.convave},
we first collect several properties of the adjacent quotient $\s_k/\s_{k-1}$
and then prove the key quantitative concavity inequality (Lemma \ref{lem.convave-pro}).
Using this inequality,
we establish Theorem \ref{thm.quotient-concavity}.
In Section \ref{sec.Jocabi},
we derive the corresponding Jacobi inequality
and, following the framework of Lu and Tsai \cite{LT26},
complete the proof of Theorem \ref{thm.main}.
Finally, in Section \ref{sec.Liouville},
we combine the constant-rank theorem, the Legendre transform,
Theorem \ref{thm.main}, and the Evans--Krylov estimate
to prove Theorem \ref{thm.liouville}.

\medskip

Throughout this paper, $C$ denotes a positive constant
which may change from line to line.
Its dependence on the fixed parameters in the context will usually be omitted,
while any additional dependence will be specified explicitly.

\section{A concavity inequality for Hessian quotient operators}\label{sec.convave}

In this section,
we aim to establish the concavity inequality in Theorem \ref{thm.quotient-concavity}.
The main step is to prove a quantitative inequality
for the adjacent quotient $\s_k/\s_{k-1}$ in Lemma \ref{lem.convave-pro},
whose proof is based on a contradiction argument.
We begin with some basic notation.
\medskip

For $n\geq1$ and $\l=(\l_1,\l_2,\dots,\l_n)\in\R^n$,
denote by $\l|i$ the vector obtained from $\l$
by deleting its $i$-th component.
Let $\s_i(\l)$ denote the $i$-th elementary symmetric function,
with the conventions $\s_0(\l):=1$, $\s_i(\l):=0$ if $i<0$ or $i>n$.
We write
\[
\G_+^n:=(0,\infty)^n,\
q_i^{(n)}(\l)
:=\frac{\s_i(\l)}{\s_{i-1}(\l)}
\quad\text{for all}\
i\in\Z_{[1,n]}.
\]
When the dimension is clear, we simply write $q_i$.
In particular, $q_1=\s_1$.

Next, we collect three properties of $q_k$,
which will be used in the proof of Lemma \ref{lem.convave-pro}.
The first shows that $q_k$ is concave
and that $\R y$ is the only null direction of its Hessian.

\begin{lemma}\label{lem.strict-concavity}
Let $2\leq k\leq n$ and
$y\in\ol{\G_+^n}$ satisfy $\s_k(y)>0$.
Then
\[
Dq_k(y)\cdot y=q_k(y),
\quad
-D^2q_k(y)\geq0,
\quad
\ker D^2q_k(y)=\R y.
\]
\end{lemma}

\begin{proof}
Since $\s_k(y)>0$ and $y\in\ol{\G_+^n}$,
the vector $y$ has at least $k$ positive components.
In particular, $\s_j(y)>0$ for all $j\in\Z_{[1,k]}$.
Since $q_k$ is homogeneous of degree one,
we have $Dq_k(y)\cdot y=q_k(y)$.
Then $D^2q_k(y)y=0$ and hence
\[
\R y\subset\ker D^2q_k(y).
\]
By \cite[Theorem~2.5]{HS99}, $q_k$ is concave and hence $-D^2q_k(y)\geq0$.
Moreover, the equality characterization in
\cite[Theorem~2.5]{HS99} shows that
\[
\p_v^2q_k(y)=0
\]
only if either $v$ is a scalar multiple of $y$,
or $y$ has exactly $n-k+1$ zero components.
The latter case is impossible since $\s_k(y)>0$.
Therefore, $\p_v^2q_k(y)=0$ implies $v\in\R y$.
Thus $\ker D^2q_k(y)=\R y$.
\end{proof}

\medskip

The following lemma compares the $k$-th largest component with $q_k$.
In particular, under the normalization $q_k(\l)=1$,
the component remains uniformly bounded above and away from zero.

\begin{lemma}\label{lem.lambda-k-bounded}
Let $1\leq k\leq n$ and let $\l\in\G_+^n$ with $\l_1\geq\l_2\geq\cdots\geq\l_n$.
Then
\[
\frac{1}{n-k+1}q_k(\l)
\leq\l_k
\leq\binom{n}{k-1}q_k(\l).
\]
\end{lemma}

\begin{proof}
Since $\s_k(\l)\geq\l_1\l_2\cdots\l_k$
and $\s_{k-1}(\l)\leq\binom{n}{k-1}\l_1\l_2\cdots\l_{k-1}$,
we have
\[
\l_k
\leq\frac{\s_k(\l)}{\l_1\l_2\cdots\l_{k-1}}
\leq\binom{n}{k-1}\frac{\s_k(\l)}{\s_{k-1}(\l)}
=\binom{n}{k-1}q_k(\l).
\]
On the other hand, every monomial in $\s_k(\l)$
contains a factor at most $\l_k$,
and deleting one such factor produces a monomial in
$\s_{k-1}(\l)$ with multiplicity at most $n-k+1$.
Hence $\s_k(\l)\leq(n-k+1)\l_k\s_{k-1}(\l)$,
which gives
\[
\l_k
\geq\frac{1}{n-k+1}\frac{\s_k(\l)}{\s_{k-1}(\l)}
=\frac{1}{n-k+1}q_k(\l).
\]
This completes the proof.
\end{proof}

\medskip

We also record the following identity for $q_k$,
obtained by replacing some components with their reciprocals.

\begin{lemma}\label{lem.qk=Phi}
Let $1\leq r<k\leq n$
and let $\l=(\l_i)\in\R^n$ satisfy $\l_i\neq 0$ for all $i\in\Z_{[1,r]}$.
Set $x:=\frb{\frac1{\l_1},\ldots,\frac1{\l_r}}$ and $y:=(\l_{r+1},\ldots,\l_n)$.
If $\s_{k-1}(\l)\neq0$, then
\[
\frac{\s_k(\l)}{\s_{k-1}(\l)}
=\frac{\sum_{i=0}^r\s_i(x)\s_{k-r+i}(y)}{\sum_{i=0}^r\s_i(x)\s_{k-r-1+i}(y)}.
\]
\end{lemma}

\begin{proof}
Set $P:=\prod_{a=1}^r\l_a$.
For any $j\in\Z_{[1,n]}$, we have
\[
\s_j(\l)
=\sum_{\ell=0}^r\s_\ell(\l_1,\ldots,\l_r)\s_{j-\ell}(y)
=P\sum_{\ell=0}^r\s_{r-\ell}(x)\s_{j-\ell}(y)
=P\sum_{i=0}^r\s_i(x)\s_{j-r+i}(y).
\]
Taking $j=k$ and $j=k-1$, and then taking the quotient,
gives the desired identity.
\end{proof}

\medskip

We are now ready to prove the key estimate for $q_k$,
using the preceding properties of the adjacent quotient.

\begin{lemma}[Key lemma]\label{lem.convave-pro}
Let $2\leq k\leq N$, $0<\d<1/2$, and $a>0$.
Set $q:=q_k$.
There exist constants $A=A(N,k,\d,a)\geq2$ and $K=K(N,k,\d,a)\geq1$
such that the following holds.
Suppose that $\l\in\R^N$ satisfies
\[
\l_1=\l_2=\cdots=\l_m=:\l_{\max}>\l_{m+1}\geq\cdots\geq \l_N>0
\]
for some $m\in\Z_{[1,N]}$, and
\[
\frac{\l_{\max}}{q(\l)}\geq A.
\]
If $c>0$ satisfies
\[
c\geq\frac{aq(\l)}{\l_{\max}},
\]
then
\begin{equation}\label{eqn.concave-pro}
-\frac{\p_\xi^2q(\l)}{q(\l)}+Kc\frb{\frac{\p_\xi q(\l)}{q(\l)}}^2
+\sum_{i>m}\frac{2\p_iq(\l)\xi_i^2}{(\l_{\max}-\l_i)q(\l)}
\geq(1+\d)\frac{\p_1q(\l)\xi_1^2}{\l_{\max}q(\l)}
\end{equation}
for every $\xi\in\R^N$ with $\xi_1=\xi_2=\cdots=\xi_m$.
When $m=N$,
the terms involving $\l_{m+1}$ and the sums over $i>m$ are understood to be omitted.
\end{lemma}

\begin{proof}
For simplicity, write $\g:=1+\d<3/2$.
The inequality \eqref{eqn.concave-pro} is invariant
under
\[
(\l,\xi)\mapsto(\wt\l,\wt\xi)
=\frb{\frac{\l}{q(\l)},\frac{\xi}{q(\l)}}.
\]
Indeed, $q$ is homogeneous of degree one,
$\p_iq$ is homogeneous of degree zero,
and $\p_{ij}q$ is homogeneous of degree $-1$.
Thus, $\p_iq(\l)=\p_iq(\wt\l)$,
$\p_\xi q(\l)=q(\l)\p_{\wt\xi}q(\wt\l)$ and
$\p_\xi^2q(\l)=\xi^\top D^2q(\l)\xi
=q(\l)\wt\xi^\top D^2q(\wt\l)\wt\xi
=q(\l)\p_{\wt\xi}^2q(\wt\l)$.
Hence, we may normalize
\[
q(\l)=1.
\]
If $\xi_1=0$, then
$\xi_1=\cdots=\xi_m=0$,
and the right-hand side of \eqref{eqn.concave-pro} vanishes.
The left-hand side is nonnegative because
$q$ is concave and $\p_iq(\l)>0$ for all $i\in\Z_{[1,N]}$.
Thus, it suffices to consider $\xi_1\neq0$.
Furthermore, since both sides of \eqref{eqn.concave-pro}
are homogeneous of degree two in $\xi$,
replacing $\xi$ by $\xi/\xi_1$ does not change the inequality.
Hence, we may assume
\[
\xi_1=\xi_2=\cdots=\xi_m=1.
\]

Suppose, to the contrary, that no such constants $A$ and $K$ exist.
Then for any $\nu\in\Z_{>1}$, taking $A=K=\nu$,
we obtain
$m_\nu\in\Z_{[1,N]}$,
$\l^\nu\in\R^N$,
$c_\nu>0$,
and $\xi^\nu\in\R^N$
such that
\begin{gather}
q(\l^\nu)=1,\
\l_1^\nu=\cdots=\l_{m_\nu}^\nu
=:\l_{\max}^\nu
>\l_{m_\nu+1}^\nu
\geq\cdots\geq
\l_N^\nu>0,\
\nonumber\\
\xi_1^\nu=\xi_2^\nu=\cdots=\xi_{m_\nu}^\nu=1,
\nonumber\\
\l_{\max}^\nu\geq \nu,
\quad
c_\nu\geq \frac{a}{\l_{\max}^\nu},
\label{eqn.l-max-lower}
\end{gather}
and
\begin{equation}\label{eqn.contradiction-sequence}
-\p_{\xi^\nu}^2q(\l^\nu)
+\nu c_\nu\frb{\p_{\xi^\nu}q(\l^\nu)}^2
+\sum_{i>m_\nu}\frac{2\p_iq(\l^\nu)(\xi_i^\nu)^2}{\l_{\max}^\nu-\l_i^\nu}
<\g\frac{\p_1q(\l^\nu)}{\l_{\max}^\nu}.
\end{equation}
Since $m_\nu\in\Z_{[1,N]}$, after passing to a subsequence,
we may assume that $m_\nu=m$ for some fixed $m\in\Z_{[1,N]}$.
Moreover, since $\l_{\max}^\nu\to+\infty$ $(\nu\to+\infty)$,
while $\l_k^\nu$ is uniformly bounded above by
Lemma \ref{lem.lambda-k-bounded},
we necessarily have $m\leq k-1$.
Passing to a further subsequence,
and using the monotonicity of the coordinates,
we may assume that there exists $m\leq r\leq k-1$
such that
\begin{equation}\label{eqn.live-infty-r}
\l_i^\nu\to+\infty\ (\nu\to+\infty)
\quad\text{for all}\
1\leq i\leq r,
\end{equation}
while
\begin{equation}\label{eqn.yve-infty-N-r}
y^\nu
=(y_i^\nu)_{i=r+1}^N
:=(\l_i^\nu)_{i=r+1}^N
\to y=(y_i)_{i=r+1}^N\in[0,\infty)^{N-r}.
\end{equation}
Set
\[
p:=k-r,
\quad
Q_p(z):=\frac{\s_p(z)}{\s_{p-1}(z)}.
\]
Since $\l_k^\nu$ is the $p$-th largest component of $y^\nu$
and $\l_k^\nu\geq\frac{1}{N-k+1}$ by Lemma \ref{lem.lambda-k-bounded},
we have $\s_p(y)>0$.

\medskip

Next, write
$x^\nu:=\frb{\frac1{\l_1^\nu},\ldots,\frac1{\l_r^\nu}}\in\R^r$,
$v_i^\nu:=\frac{\xi_i^\nu}{x_1^\nu}$ for all $i\in\Z_{[r+1,N]}$
and $v^\nu:=(v_i^\nu)_{i=r+1}^N\in\R^{N-r}$.
We claim that the sequence $\set{v^\nu}$ is bounded, and after passing to a further subsequence,
\[
v^\nu\to v\quad(\nu\to+\infty)
\]
for some $v\in\R^{N-r}$ satisfying
\begin{equation}\label{eqn.limit-tangent}
DQ_p(y)\cdot v=0.
\end{equation}
To see this, consider the function
\[
\Phi(x,z):=\frac{\sum_{i=0}^r\s_i(x)\s_{p+i}(z)}{\sum_{i=0}^r\s_i(x)\s_{p-1+i}(z)}.
\]
By Lemma \ref{lem.qk=Phi},
\eqref{eqn.live-infty-r} and \eqref{eqn.yve-infty-N-r},
we have
\begin{equation}\label{eqn.Phi-def}
1=q(\l^\nu)
=\Phi(x^\nu,y^\nu)
\to\Phi(0,y)
=\frac{\s_p(y)}{\s_{p-1}(y)}
=Q_p(y)
\quad
(\nu\to+\infty).
\end{equation}
In particular, $Q_p(y)=1$.
We can compute
\begin{gather*}
\a:=-\p_{x_i}\Phi(0,y)
=\frac{\s_p(y)^2-\s_{p-1}(y)\s_{p+1}(y)}{\s_{p-1}(y)^2}
>0
\quad\text{for all}\ i\in\Z_{[1,r]},
\\
\p_{y_i}\Phi(0,y)
=\p_{y_i}Q_p(y)
=\frac{(\s_{p-1}(y|i))^2-\s_p(y|i)\s_{p-2}(y|i)}{\s_{p-1}(y)^2}
>0
\quad\text{for all}\ i\in\Z_{[r+1,N]},
\end{gather*}
where the strict inequalities follow from the Newton inequalities
and the fact that $\s_p(y)>0$.
In what follows, all $O(\cdot)$ and $o(\cdot)$ terms are understood
as $\nu\to+\infty$.

Since $\Phi$ is smooth in a neighborhood of $(0,y)$,
and $(x^\nu,y^\nu)\to(0,y)$ $(\nu\to+\infty)$,
we have
\begin{align*}
\p_{x_i}\Phi(x^\nu,y^\nu)
=\p_{x_i}\Phi(0,y)+o(1)
=-\a+o(1)
\quad
\text{for any}\ i\in\Z_{[1,r]},
\qquad
D^2\Phi(x^\nu,y^\nu)=O(1).
\end{align*}
Using $\frac{\p x_i}{\p\l_i}=-x_i^2$, we deduce
\begin{align}
\p_{\l_i}q(\l^\nu)
&=-(x_i^\nu)^2\p_{x_i}\Phi(x^\nu,y^\nu)
=\a(x_i^\nu)^2+o((x_i^\nu)^2)
\quad
\text{for any}\  i\in\Z_{[1,r]},
\label{eqn.qi-r}
\\
\p_{\l_i}q(\l^\nu)
&=\p_{y_i}\Phi(x^\nu,y^\nu)
\to\p_{y_i}\Phi(0,y)
=\p_{y_i}Q_p(y)>0
\quad
\text{for any}\  i\in\Z_{[r+1,N]}.
\label{eqn.qi-N}
\end{align}
Recalling that $x_i^\nu=x_1^\nu$, $\xi_i^\nu=1$ for all $i\in\Z_{[1,m]}$,
we find
\begin{equation}\label{eqn.qixi-1-m}
\frac{\p_iq(\l^\nu)}{x_1^\nu}
=\a x_1^\nu+o(x_1^\nu)
\to0
\quad
\text{for any}\ i\in\Z_{[1,m]}.
\end{equation}
Since every term on the left-hand side of
\eqref{eqn.contradiction-sequence} is nonnegative,
\[
\frac{2\p_iq(\l^\nu)(\xi_i^\nu)^2}{\l_{\max}^\nu}
\leq\frac{2\p_iq(\l^\nu)(\xi_i^\nu)^2}{\l_{\max}^\nu-\l_i^\nu}
<\g\frac{\p_1q(\l^\nu)}{\l_{\max}^\nu}
\quad
\text{for any}\ i\in\Z_{[m+1,N]}.
\]
Combining this with \eqref{eqn.qi-r} and \eqref{eqn.qi-N},
for sufficiently large $\nu$, we obtain
\begin{align}
\abs{\xi_i^\nu}^2
&\leq\frac{\g\p_1q(\l^\nu)}{2\p_iq(\l^\nu)}
\leq\frac{3\g}{2}\frb{\frac{x_1^\nu}{x_i^\nu}}^2
\quad
\text{for any}\ i\in\Z_{[m+1,r]},
\label{eqn.xin-m-r}
\\
\abs{\xi_i^\nu}^2
&\leq\frac{\g\p_1q(\l^\nu)}{2\p_iq(\l^\nu)}
\leq\frac{3\g\a}{2\min\limits_{i\in\Z_{[r+1,N]}}\p_{y_i}Q_p(y)}(x_1^\nu)^2
\quad
\text{for any}\ i\in\Z_{[r+1,N]}.
\label{eqn.xin-r-N}
\end{align}
Hence for any $i\in\Z_{[m+1,r]}$, we have
\begin{equation}\label{eqn.qixi-m-r}
\abs{\frac{\p_iq(\l^\nu)\xi_i^\nu}{x_1^\nu}}
\leq\frac{2\a(x_i^\nu)^2}{x_1^\nu}\sqrt{\frac{3\g}{2}}\,
\frac{x_1^\nu}{x_i^\nu}
=2\a\sqrt{\frac{3\g}{2}}\,x_i^\nu
\to0
\quad
(\nu\to+\infty).
\end{equation}
Moreover, \eqref{eqn.xin-r-N} implies that $\set{v^\nu}$ is bounded.
Passing to a further subsequence, we may assume that
$v^\nu\to v$ $(\nu\to+\infty)$ for some $v=(v_i)\in\R^{N-r}$.
By \eqref{eqn.qi-N},
\[
\sum_{i=r+1}^N\p_iq(\l^\nu)v_i^\nu
\to\sum_{i=r+1}^N\p_{y_i}Q_p(y)v_i
=DQ_p(y)\cdot v
\quad
(\nu\to+\infty).
\]
Combining this with \eqref{eqn.qixi-1-m}
and \eqref{eqn.qixi-m-r}, we obtain
\[
\frac{\p_{\xi^\nu}q(\l^\nu)}{x_1^\nu}
=\sum_{i=1}^m\frac{\p_iq(\l^\nu)}{x_1^\nu}
+\sum_{i=m+1}^r
\frac{\p_iq(\l^\nu)\xi_i^\nu}{x_1^\nu}
+\sum_{i=r+1}^N\p_iq(\l^\nu)v_i^\nu
\to DQ_p(y)\cdot v
\quad
(\nu\to+\infty).
\]
It remains to identify this limit.
The second term in \eqref{eqn.contradiction-sequence},
together with \eqref{eqn.qi-r}, gives
\[
\nu c_\nu\frb{\p_{\xi^\nu}q(\l^\nu)}^2
<\g\frac{\p_1q(\l^\nu)}{\l_{\max}^\nu}
\leq2\a\g(x_1^\nu)^3
\]
for sufficiently large $\nu$.
By $c_\nu\geq\frac{a}{\l_{\max}^\nu}=ax_1^\nu$, we have
\[
\frac{\abs{\p_{\xi^\nu}q(\l^\nu)}}{x_1^\nu}
\leq\frac{\sqrt{2\a\g x_1^\nu}}{\sqrt{\nu c_\nu}}
\leq\sqrt{\frac{2\a\g}{a}}\frac{1}{\sqrt\nu}
\to0
\quad
(\nu\to+\infty).
\]
Therefore, $DQ_p(y)\cdot v=0$,
which proves the claim.

\medskip

We now use the second-order expansion of $q$
to get the contradiction.
We can compute
\begin{align}
\p_{\l_i\l_i}q(\l^\nu)
&=2(x_i^\nu)^3\p_{x_i}\Phi(x^\nu,y^\nu)
+(x_i^\nu)^4\p_{x_ix_i}\Phi(x^\nu,y^\nu)
\nonumber
\\
&=2(x_i^\nu)^3\frb{-\a+o(1)}
+O\frb{(x_i^\nu)^4}
=-2\a(x_i^\nu)^3
+o\frb{(x_i^\nu)^3},
\quad
\A i\in\Z_{[1,r]},
\label{eqn.qii-r}
\\
\p_{\l_i\l_j}q(\l^\nu)
&=(x_i^\nu)^2(x_j^\nu)^2
\p_{x_ix_j}\Phi(x^\nu,y^\nu)
=O\frb{(x_i^\nu)^2(x_j^\nu)^2},
\quad
\A i,j\in\Z_{[1,r]},\ i\neq j,
\label{eqn.qij-r}
\\
\p_{\l_i\l_j}q(\l^\nu)
&=-(x_i^\nu)^2\p_{x_iy_j}\Phi(x^\nu,y^\nu)
=O\frb{(x_i^\nu)^2},
\quad
\A i\in\Z_{[1,r]},\
j\in\Z_{[r+1,N]},
\label{eqn.qij-rN}
\\
\p_{\l_i\l_j}q(\l^\nu)
&=\p_{y_iy_j}\Phi(x^\nu,y^\nu)
\to
\p_{y_iy_j}\Phi(0,y)
=\p_{y_iy_j}Q_p(y),
\quad
\A i,j\in\Z_{[r+1,N]}.
\label{eqn.qij-N}
\end{align}
By \eqref{eqn.contradiction-sequence}
and the concavity of $q$, together with \eqref{eqn.qi-r},
we have
\begin{equation}\label{eqn.q-xi-2}
0\leq-\p_{\xi^\nu}^2q(\l^\nu)
<\g\frac{\p_1q(\l^\nu)}{\l_{\max}^\nu}
=\g\frb{\a(x_1^\nu)^3
+o\frb{(x_1^\nu)^3}}.
\end{equation}
We divide the remaining argument into two cases according to $p$.

\medskip

\textit{Case 1: $p\geq2$}.\
We first show that
\[
-\p_v^2Q_p(y)=0.
\]
Indeed, \eqref{eqn.q-xi-2} gives
\[
\frac{-\p_{\xi^\nu}^2q(\l^\nu)}{(x_1^\nu)^2}
\to0
\quad
(\nu\to+\infty).
\]
On the other hand, we decompose
\begin{align}
\frac{\p_{\xi^\nu}^2q(\l^\nu)}{(x_1^\nu)^2}
=\sum_{i,j=1}^r\p_{ij}q(\l^\nu)\frac{\xi_i^\nu\xi_j^\nu}{(x_1^\nu)^2}
+2\sum_{i=1}^r\sum_{j=r+1}^N\p_{ij}q(\l^\nu)\frac{\xi_i^\nu\xi_j^\nu}{(x_1^\nu)^2}
+\sum_{i,j=r+1}^N\p_{ij}q(\l^\nu)\frac{\xi_i^\nu\xi_j^\nu}{(x_1^\nu)^2}.
\label{eqn.hessian-block-decomposition}
\end{align}
By \eqref{eqn.qii-r}, \eqref{eqn.qij-r},
\eqref{eqn.xin-m-r},
$x_i^\nu=x_1^\nu$ and $\xi_i^\nu=1$
for all $i\in\Z_{[1,m]}$,
we have
\begin{align}
\abs{\sum_{i,j=1}^r\p_{ij}q(\l^\nu)\frac{\xi_i^\nu\xi_j^\nu}{(x_1^\nu)^2}}
&\leq\sum_{i=1}^r\abs{\p_{ii}q(\l^\nu)}\frac{(\xi_i^\nu)^2}{(x_1^\nu)^2}
+2\sum_{1\leq i<j\leq r}\abs{\p_{ij}q(\l^\nu)}
\frac{\abs{\xi_i^\nu\xi_j^\nu}}{(x_1^\nu)^2}
\nonumber\\
&\leq C\frb{x_1^\nu+\sum_{i=m+1}^rx_i^\nu+(x_1^\nu)^2
+\sum_{i=m+1}^rx_1^\nu x_i^\nu
+\sum_{i=m+1}^r(x_i^\nu)^2}
\nonumber\\
&\to0
\quad
(\nu\to+\infty).
\label{eqn.qij-x1n-1-r}
\end{align}
Since $\xi_j^\nu=x_1^\nu v_j^\nu$
for all $j\in\Z_{[r+1,N]}$ and $v^\nu$ is bounded,
\eqref{eqn.xin-m-r} and \eqref{eqn.qij-rN} give
\begin{equation}\label{eqn.qij-x1n-1-r-r-N}
\abs{\sum_{i=1}^r\sum_{j=r+1}^N\p_{ij}q(\l^\nu)
\frac{\xi_i^\nu\xi_j^\nu}{(x_1^\nu)^2}}
\leq
C\sum_{j=r+1}^N\frb{x_1^\nu+\sum_{i=m+1}^rx_i^\nu}\abs{v_j^\nu}
\to0
\quad
(\nu\to+\infty).
\end{equation}
Moreover, by \eqref{eqn.qij-N} and $v^\nu\to v$,
\begin{equation}\label{eqn.qij-x1n-r-N}
\sum_{i,j=r+1}^N\p_{ij}q(\l^\nu)\frac{\xi_i^\nu\xi_j^\nu}{(x_1^\nu)^2}
\to\sum_{i,j=r+1}^N
\p_{y_iy_j}Q_p(y)v_iv_j
=\p_v^2Q_p(y)
\quad
(\nu\to+\infty).
\end{equation}
Combining
\eqref{eqn.qij-x1n-1-r}--\eqref{eqn.qij-x1n-r-N},
we obtain $-\p_v^2Q_p(y)=0$.

Since $p\geq2$ and $\s_p(y)>0$,
Lemma \ref{lem.strict-concavity} gives
$v=\b y$ for some $\b\in\R$.
By \eqref{eqn.limit-tangent},
\eqref{eqn.Phi-def}, and Euler's identity,
\[
0=DQ_p(y)\cdot v
=\b DQ_p(y)\cdot y
=\b Q_p(y)
=\b.
\]
Thus, $v=0$, and hence
\begin{equation}\label{eqn.xi-tail-little-o}
\xi_i^\nu=o(x_1^\nu)
\quad
\text{for all}\ i\in\Z_{[r+1,N]}.
\end{equation}
Next, by $-D^2q(\l^\nu)\geq0$
and the estimates \eqref{eqn.xin-r-N},
\eqref{eqn.qii-r}--\eqref{eqn.qij-rN}
and \eqref{eqn.xi-tail-little-o},
we have
\begin{align*}
-\p_{\xi^\nu}^2q(\l^\nu)
&\geq-\p_{11}q(\l^\nu)-2\sum_{j=2}^N\abs{\p_{1j}q(\l^\nu)\xi_j^\nu}
%%\\
%%&\geq2\a(x_1^\nu)^3+o\frb{(x_1^\nu)^3}
%%-2\sum_{j=2}^mO\frb{(x_1^\nu)^2(x_j^\nu)^2}
%%\\
%%&\quad\
%%-2C\sum_{j=m+1}^r(x_1^\nu)^2(x_j^\nu)^2\frac{x_1^\nu}{x_j^\nu}
%%-O\frb{(x_1^\nu)^2}o(x_1^\nu)
%%\\
\geq2\a(x_1^\nu)^3+o\frb{(x_1^\nu)^3}.
\end{align*}
Combining this with $\frac{\p_1q(\l^\nu)}{\l_{\max}^\nu}=\a(x_1^\nu)^3+o\frb{(x_1^\nu)^3}$,
we conclude
\begin{equation}\label{eqn.ratio-two}
\liminf_{\nu\to+\infty}
\frac{-\p_{\xi^\nu}^2q(\l^\nu)}{\p_1q(\l^\nu)/\l_{\max}^\nu}\geq2.
\end{equation}
This contradicts \eqref{eqn.q-xi-2}
since $\g<\frac32<2$.

\medskip

\textit{Case 2: $p=1$}.\
In this case, $r=k-1$.
By \eqref{eqn.Phi-def} and \eqref{eqn.limit-tangent}, it follows that
\[
Q_1(y)=\s_1(y)=1,
\quad
\sum_{i=r+1}^Nv_i
=DQ_1(y)\cdot v
=0,
\]
and
\[
\a
=-\p_{x_i}\Phi(0,y)
=\s_1(y)^2-\s_2(y)
=\s_1(y)^2
-\frb{\frac{\s_1(y)^2-|y|^2}{2}}
=\frac{1+\abs{y}^2}{2}
\quad
\text{for all}\ i\in\Z_{[1,r]}.
\]
In this limiting linear case,
the mixed terms must be kept at the scale $(x_1^\nu)^3$.
We have
\[
-\p_{x_1y_i}\Phi(0,y)
=\p_{y_i}\frb{\s_1(y)^2-\s_2(y)}
=\s_1(y)+y_i
=1+y_i.
\]
By \eqref{eqn.qij-rN},
\[
\p_{\l_1\l_i}q(\l^\nu)
=-(x_1^\nu)^2
\p_{x_1y_i}\Phi(x^\nu,y^\nu)
=(x_1^\nu)^2(1+y_i)
+o\frb{(x_1^\nu)^2},
\quad
i\in\Z_{[r+1,N]}.
\]
Combining this with $-D^2q(\l^\nu)\geq0$,
\eqref{eqn.qii-r}, \eqref{eqn.qij-r}
and \eqref{eqn.xin-m-r},
we obtain
\begin{align*}
-\p_{\xi^\nu}^2q(\l^\nu)
&\geq-\p_{11}q(\l^\nu)
-2\sum_{i=2}^r\abs{\p_{1i}q(\l^\nu)\xi_i^\nu}
-2\sum_{i=r+1}^N\p_{1i}q(\l^\nu)\xi_i^\nu
\\
&=2\a(x_1^\nu)^3-2(x_1^\nu)^3\sum_{i=r+1}^N(1+y_i)v_i^\nu
+o\frb{(x_1^\nu)^3}.
\end{align*}
Moreover, it follows from \eqref{eqn.qi-N} that
\[
\p_iq(\l^\nu)
\to\p_{y_i}Q_1(y)=1
\quad\text{for all}\ i\in\Z_{[r+1,N]}.
\]
Thus
\[
\frac{2\p_iq(\l^\nu)(\xi_i^\nu)^2}{\l_{\max}^\nu-\l_i^\nu}
=\frac{2\p_iq(\l^\nu)(x_1^\nu)^3}{1-x_1^\nu\l_i^\nu}(v_i^\nu)^2
=2(x_1^\nu)^3(v_i^\nu)^2+o\frb{(x_1^\nu)^3}
\quad\text{for all}\ i\in\Z_{[r+1,N]}.
\]
Consequently,
\begin{align}
\liminf_{\nu\to+\infty}
\frac{1}{(x_1^\nu)^3}
\frb{-\p_{\xi^\nu}^2q(\l^\nu)+\sum_{i=m+1}^N
\frac{2\p_iq(\l^\nu)(\xi_i^\nu)^2}
{\l_{\max}^\nu-\l_i^\nu}}
&\geq2\a-2\sum_{i=r+1}^N(1+y_i)v_i+2|v|^2
\notag\\
&=2\a-2y\cdot v+2|v|^2.
\label{eqn.p1-lower}
\end{align}
Let $\mathbf 1:=(1,\ldots,1)\in\R^{N-r}$
and $y^\perp:=y-\frac{1}{N-r}\mathbf 1$.
Since $v\cdot\mathbf 1=0$, we have
\[
y\cdot v=y^\perp\cdot v.
\]
Thus
\begin{equation}\label{eqn.2a>}
2\a-2y\cdot v+2\abs{v}^2
=2\a+2\abs{v-\frac12y^\perp}^2
-\frac12\abs{y^\perp}^2
\geq2\a-\frac12\abs{y^\perp}^2
\geq\frac32\a,
\end{equation}
where the last inequality follows from
$\s_1(y)=1$ and $y\geq0$, which give $|y|^2\leq1$, and
\[
\abs{y^\perp}^2
=|y|^2-\frac{|\mathbf 1|^2}{(N-r)^2}
\leq\frac{1+|y|^2}{2}
=\a.
\]
Combining \eqref{eqn.q-xi-2}, \eqref{eqn.p1-lower} and \eqref{eqn.2a>}, we conclude that
\[
\liminf_{\nu\to+\infty}
\frac{1}{\p_1q(\l^\nu)/\l_{\max}^\nu}
\frb{-\p_{\xi^\nu}^2q(\l^\nu)
+\sum_{i=m+1}^N\frac{2\p_iq(\l^\nu)(\xi_i^\nu)^2}{\l_{\max}^\nu-\l_i^\nu}}
\geq\frac32
>\g.
\]
However, after dropping the nonnegative gradient term in \eqref{eqn.contradiction-sequence},
\begin{equation}\label{eqn.q-xi-gap}
0
\leq-\p_{\xi^\nu}^2q(\l^\nu)
+\sum_{i=m+1}^N\frac{2\p_iq(\l^\nu)(\xi_i^\nu)^2}{\l_{\max}^\nu-\l_i^\nu}
<\g\frac{\p_1q(\l^\nu)}{\l_{\max}^\nu},
\end{equation}
a contradiction.
The proof is complete.
\end{proof}

\medskip

With the key inequality established,
we now prove Theorem \ref{thm.quotient-concavity}
for the Hessian quotient $F(\l)=\frac{\s_k(\l)}{\s_l(\l)}$.
When $k-l=1$, $F$ is already an adjacent quotient.
When $k-l=2$, we reduce the problem to an adjacent quotient in one higher dimension.

\begin{proof}[Proof of Theorem \ref{thm.quotient-concavity}]
Let $\wt  A$ denote the constant $A$
in Lemma \ref{lem.convave-pro} corresponding to
\[
(N,a)
=\begin{cases}
(n,1),&k-l=1,\\
\frb{n+1,\frac{2(k-1)}{n+1}},&k-l=2.
\end{cases}
\]
Take $A:=2\wt  A$.
We first assume that $\l_n>0$ and consider two cases.

\medskip

\textit{Case 1: $k-l=1$.}
We have $F=q_k$.
Applying Lemma \ref{lem.convave-pro} with $N=n$ and $a=c=1$,
the assumption $\l_{\max}\geq AF(\l)$ gives
$\frac{\l_{\max}}{F(\l)}\geq A\geq\wt  A$
and $\frac{F(\l)}{\l_{\max}}\leq1$.
Hence the hypotheses of Lemma \ref{lem.convave-pro} are satisfied.
Multiplying \eqref{eqn.concave-pro} by $F(\l)>0$
gives the desired inequality.

\medskip

\textit{Case 2: $k-l=2$.}
We have $F(\l)=\frac{\s_k(\l)}{\s_{k-2}(\l)}$.
Let $\L:=\frb{\l_1,\ldots,\l_n,\sqrt{F(\l)}}\in\R^{n+1}$.
Then $\L\in\G_+^{n+1}$.
Write $q^*:=q_k^{(n+1)}$. We compute
\begin{align}
q^*(\L)
&=\frac{\s_k(\l)+\sqrt{F(\l)}\s_{k-1}(\l)}{\s_{k-1}(\l)+\sqrt{F(\l)}\s_{k-2}(\l)}
=\sqrt{F(\l)}.
\label{eqn.qk-n+1}\\
\p_{n+1}q^*(\L)
&=\frac{\s_{k-1}(\l)^2-\s_k(\l)\s_{k-2}(\l)}{\frb{\s_{k-1}(\l)+\sqrt{F(\l)}\s_{k-2}(\l)}^2}
=\frac{\s_{k-1}(\l)-\sqrt{F(\l)}\s_{k-2}(\l)}{\s_{k-1}(\l)+\sqrt{F(\l)}\s_{k-2}(\l)}
\in[0,1).
\label{eqn.Dqk-n+1}
\end{align}
Moreover, we have $(k-1)\s_{k-1}(\l)
=\sum_{i=1}^n\l_i\s_{k-2}(\l|i)
\leq(n-k+2)\l_{\max}\s_{k-2}(\l)$,
and hence
\[
q_{k-1}(\l)
\leq\frac{n-k+2}{k-1}\l_{\max}.
\]
Since $\l_{\max}\geq A\sqrt{F(\l)}>\sqrt{F(\l)}$,
it follows from \eqref{eqn.Dqk-n+1} that
\begin{equation}\label{eqn.c-lower}
1-\p_{n+1}q^*(\L)
=\frac{2\sqrt{F(\l)}}{q_{k-1}(\l)+\sqrt{F(\l)}}
\geq\frac{2\sqrt{F(\l)}}{\frb{\frac{n-k+2}{k-1}+1}\l_{\max}}
=\frac{2(k-1)}{n+1}\frac{\sqrt{F(\l)}}{\l_{\max}}.
\end{equation}
For any $\xi\in\R^n$ with $\xi_1=\xi_2=\cdots=\xi_m$,
set $\Xi:=\frb{\xi,\frac{\p_\xi F(\l)}{2\sqrt{F(\l)}}}\in\R^{n+1}$.
After a simultaneous permutation of the last $n+1-m$ components of $\L$ and $\Xi$,
we may assume that these components of $\L$ are in decreasing order.
By symmetry, Lemma \ref{lem.convave-pro} applies,
and we still denote the added coordinate by $n+1$.
Take
\[
a=\frac{2(k-1)}{n+1},\quad
c=1-\p_{n+1}q^*(\L).
\]
By \eqref{eqn.qk-n+1}, \eqref{eqn.c-lower},
and the choice $A=2\wt  A$,
the hypotheses of Lemma \ref{lem.convave-pro} are satisfied.
Thus there exists $K=K(N,k,\d,a)\geq1$ such that
\begin{align}
-\frac{\p_\Xi^2q^*(\L)}{q^*(\L)}
+Kc\frb{\frac{\p_\Xi q^*(\L)}{q^*(\L)}}^2
&+\sum_{i=m+1}^{n+1}\frac{2\p_iq^*(\L)\Xi_i^2}{(\L_{\max}-\L_i)q^*(\L)}
\nonumber\\
&\quad\geq(1+\d)\frac{\p_1q^*(\L)\Xi_1^2}{\L_{\max}q^*(\L)}
\label{eqn.concave-n+1}
\end{align}
for every such $\Xi$.

Next, we simplify the above inequality.
Differentiating \eqref{eqn.qk-n+1} with respect to $\l_i$ gives
\begin{equation}\label{eqn.qk-N-i}
\p_iq^*(\L)
=\frb{1-\p_{n+1}q^*(\L)}\frac{F_i}{2\sqrt{F}}
=c\frac{F_i}{2\sqrt{F}},
\quad\text{for all}\ i\in\Z_{[1,n]}.
\end{equation}
Moreover, differentiating
\[
q^*\frb{\l,\sqrt{F(\l)}}=\sqrt{F(\l)}
\]
once and twice in the direction $\xi$,
we have
\begin{gather}
\p_\Xi q^*(\L)
=\frac{\p_\xi F}{2\sqrt F},
\label{eqn.qk-xi-N-i}
\\
\p_\Xi^2q^*(\L)
+\p_{n+1}q^*(\L)\p_\xi\frb{\frac{\p_\xi F}{2\sqrt F}}
=\p_\xi\frb{\frac{\p_\xi F}{2\sqrt F}}.
\nonumber
\end{gather}
By $\p_\xi\frb{\frac{\p_\xi F}{2\sqrt F}}
=\frac{\p_\xi^2F}{2\sqrt F}-\frac{(\p_\xi F)^2}{4F^{3/2}}$,
it follows that
\begin{equation}\label{eqn.qk-xi-N}
\p_\Xi^2q^*(\L)
=c\frb{\frac{\p_\xi^2F}{2\sqrt F}-\frac{(\p_\xi F)^2}{4F^{3/2}}}.
\end{equation}
Multiplying the inequality \eqref{eqn.concave-n+1}
by $\frac{2F(\l)}{c}$
and using \eqref{eqn.qk-N-i}--\eqref{eqn.qk-xi-N},
we obtain
\begin{align*}
-\p_\xi^2F
+\frac{K+1}{2F}(\p_\xi F)^2
+\sum_{i=m+1}^n
\frac{2F_i\xi_i^2}{\l_{\max}-\l_i}
&+\frac{\p_{n+1}q^*(\L)(\p_\xi F)^2}{c\frb{\l_{\max}-\sqrt F}\sqrt F}
\geq(1+\d)\frac{F_1\xi_1^2}{\l_{\max}}.
\end{align*}
Since $\l_{\max}\geq A\sqrt F\geq2\sqrt F$,
we have $\l_{\max}-\sqrt F\geq\frac12\l_{\max}$.
Combining this with $\p_{n+1}q^*(\L)\in[0,1)$ and \eqref{eqn.c-lower} gives
\[
0
\leq\frac{\p_{n+1}q^*(\L)}{1-\p_{n+1}q^*(\L)}
\frac{(\p_\xi F)^2}{\sqrt F(\l_{\max}-\sqrt F)}
\leq\frac{n+1}{2(k-1)}\frac{\l_{\max}}{\sqrt{F(\l)}}
\frac{2(\p_\xi F)^2}{\sqrt F\l_{\max}}
=\frac{n+1}{(k-1)}\frac{(\p_\xi F)^2}{F}.
\]
Therefore, after enlarging $K$ if necessary,
we obtain the inequality \eqref{eqn.quotient-concavity}.
This proves the result when $\l_n>0$.

It remains to remove the assumption $\l_n>0$.
For $\ve>0$, set $\l^\ve:=\l+\ve\mathbf 1$.
Since $\s_k(\l)>0$ and $\l\geq0$, we have $\s_l(\l)>0$.
Thus $F$, $DF$, and $D^2F$ are continuous near $\l$.
Recall that the preceding argument holds with the constant $\wt A$,
whereas $A=2\wt A$ in the statement of the theorem.
Hence $\l_{\max}\geq2\wt A F(\l)^{1/(k-l)}$
gives
\[
\l_{\max}+\ve\geq\wt A F(\l^\ve)^{1/(k-l)}
\]
for all sufficiently small $\ve>0$.
Since $\l^\ve\in\G_+^n$ and the multiplicity of its largest component
is still $m$, the preceding argument applies to $\l^\ve$
with the same $\xi$ and constants independent of $\ve$.
Letting $\ve\to0$ and using the continuity of
$F$, $DF$, and $D^2F$ implies
\eqref{eqn.quotient-concavity} at $\l$.
This completes the proof.
\end{proof}

\section{The Jacobi inequality and the Hessian estimate}\label{sec.Jocabi}

In this section, for the operator $F$, we use the notation
\[
F^{ij}
:=\frac{\p F}{\p u_{ij}}(D^2u),
\quad
F^{pq,rs}
:=\frac{\p^2F}{\p u_{pq}\p u_{rs}}(D^2u).
\]
When $D^2u$ is diagonalized with eigenvalues
$\l_1,\ldots,\l_n$, we also write $F_i:=\frac{\p F}{\p\l_i}$ and $F_{ij}:=\frac{\p^2F}{\p\l_i\p\l_j}$
so that $F^{ij}=F_i\d_{ij}$.

We now combine Theorem \ref{thm.quotient-concavity}
with the standard largest-eigenvalue calculation
(see also \cite[Lemma 4.1]{Lu25})
to obtain the following inequality.

\begin{proposition}[Jacobi inequality]\label{thm.jacobi}
Let $n\geq2$, $1\leq l<k\leq n$ with
$k-l\in\{1,2\}$.
Let $f\in C^{1,1}(B_{10}\times\R)$ be a positive function.
Suppose that $u\in C^4(B_{10})$ is a convex solution of
\[
F(D^2u)
%:=\frac{\s_k(D^2u)}{\s_l(D^2u)}
=f(x,u)
\quad\textrm{in}\ B_{10}.
\]
Let $M>0$ satisfy $\|u\|_{L^\infty(B_9)}\leq M$
and set $b(x):=\log\l_{\max}(D^2u(x))$.
Then for any $\d\in(0,\frac12)$, at every point $x\in B_9$ with
\[
\l_{\max}(D^2u)\geq A\frb{\max_{\ol{B_9}\times[-M,M]}f}^{1/(k-l)},
\]
we have
\begin{equation}\label{eqn.jacobi}
F^{ij}b_{ij}\geq\d F^{ij}b_i b_j-C
\end{equation}
in the viscosity sense,
where $A$ is the constant in Theorem \ref{thm.quotient-concavity}
and
\[
C=C\frb{n,k,l,\d,\|u\|_{C^{0,1}(B_9)},\min_{\ol{B_9}\times[-M,M]}f,\|f\|_{C^{1,1}(\ol{B_9}\times[-M,M])}}>0.
\]
\end{proposition}

\begin{proof}
Fix $x_0\in B_9$ such that
\[
\l_{\max}(D^2u(x_0))
\geq
A\frb{\max_{B_9\times[-M,M]}f}^{1/(k-l)}.
\]
After a rotation of coordinates, we may assume that
\[
D^2u(x_0)=\op{diag}(\l_1,\ldots,\l_n),
\qquad
\l_1=\cdots=\l_m=\l_{\max}>
\l_{m+1}\geq\cdots\geq\l_n\geq0.
\]
All quantities below are evaluated at $x_0$.
By \cite[Lemma~5]{BCD17}, it follows that
\begin{gather}
\d_{k\ell}(\l_1)_i=u_{k\ell i},
\quad\text{for all}\
k,\ell\in\Z_{[1,m]},\ i\in\Z_{[1,n]}
\label{eqn.lambda-max-derivative}
\\
(\l_1)_{ii}
\geq u_{11ii}+2\sum_{p>m}\frac{u_{1pi}^2}{\l_1-\l_p}
\quad\text{for all}\ i\in\Z_{[1,n]}
\label{eqn.lambda-max-second}
\end{gather}
in the viscosity sense.
Moreover, we can compute
\begin{gather*}
u_{11i}=\l_1b_i\quad\text{for all}\ i\in\Z_{[m+1,n]},\
\\
b_{ii}
\geq\frac{u_{11ii}}{\l_1}
+2\sum_{p>m}\frac{u_{1pi}^2}{\l_1(\l_1-\l_p)}-\frac{u_{11i}^2}{\l_1^2}
\quad\
\text{for all}\ i\in\Z_{[1,n]}
\end{gather*}
in the viscosity sense.
Hence
\[
F^{ij}b_{ij}
\geq\frac{1}{\l_1}\sum_{i=1}^nF_i u_{11ii}
+2\sum_{i=1}^n\sum_{p>m}\frac{F_i u_{1pi}^2}{\l_1(\l_1-\l_p)}
-\sum_{i=1}^n\frac{F_i u_{11i}^2}{\l_1^2}.
\]
Differentiating the equation $F(D^2u)=f(x,u)$
twice in the $x_1$-direction gives
\[
\sum_{i=1}^nF_i u_{ii11}
+\sum_{p,q,r,s=1}^nF^{pq,rs}u_{pq1}u_{rs1}
=\frb{f(x,u)}_{11}
\geq-C\l_1-C.
\]
%$\frb{f(x,u)}_{11}=f_{x_1x_1}+2f_{x_1u}u_1+f_{uu}u_1^2+f_u u_{11}$,
Thus
\begin{equation}\label{eqn.jacobi-second-lower}
F^{ij}b_{ij}
\geq-\frac{1}{\l_1}\sum_{p,q,r,s=1}^nF^{pq,rs}u_{pq1}u_{rs1}
+2\sum_{i=1}^n\sum_{p>m}\frac{F_i u_{1pi}^2}{\l_1(\l_1-\l_p)}
-\sum_{i=1}^n\frac{F_i u_{11i}^2}{\l_1^2}-C.
\end{equation}
Set $\xi_i:=u_{ii1}$ for all $i\in\Z_{[1,n]}$.
By the standard second-derivative formula for symmetric functions
of the eigenvalues, together with the concavity of $F^{1/(k-l)}$
(see \cite[Lemma~2.2]{LT26} and \cite[Theorem~2.5]{HS99}), we have
\begin{equation}\label{eqn.spectral-expansion-jacobi}
-\sum_{p,q,r,s=1}^n
F^{pq,rs}u_{pq1}u_{rs1}
=-\sum_{i,j=1}^nF_{ij}\xi_i\xi_j-\sum_{i\neq j}F^{ij,ji}u_{ij1}^2,
\end{equation}
and
\begin{equation}\label{eqn.third-derivative-lower}
-\sum_{i\neq j}F^{ij,ji}u_{ij1}^2
\geq-2\sum_{p>m}F^{1p,p1}u_{1p1}^2
=2\sum_{p>m}\frac{F_p-F_1}{\l_1-\l_p}u_{11p}^2.
\end{equation}
On the other hand,
\begin{equation}\label{eqn.third-derivative-lower-jacobi}
2\sum_{i=1}^n\sum_{p>m}
\frac{F_i u_{1pi}^2}{\l_1(\l_1-\l_p)}
\geq2\sum_{p>m}\frac{F_1u_{11p}^2+F_pu_{1pp}^2}{\l_1(\l_1-\l_p)}
=2\sum_{p>m}\frac{F_1u_{11p}^2+F_p\xi_p^2}{\l_1(\l_1-\l_p)}.
\end{equation}
Substituting
\eqref{eqn.spectral-expansion-jacobi}--%
\eqref{eqn.third-derivative-lower-jacobi}
into \eqref{eqn.jacobi-second-lower}, we find
\begin{align*}
F^{ij}b_{ij}
&\geq\frac{1}{\l_1}
\frb{-\sum_{i,j=1}^nF_{ij}\xi_i\xi_j+2\sum_{p>m}\frac{F_p-F_1}{\l_1-\l_p}u_{11p}^2}
+2\sum_{p>m}\frac{F_1u_{11p}^2+F_p\xi_p^2}{\l_1(\l_1-\l_p)}
-\sum_{i=1}^n\frac{F_i u_{11i}^2}{\l_1^2}-C.
\\
&=\frac{1}{\l_1}
\frb{-\sum_{i,j=1}^nF_{ij}\xi_i\xi_j+\sum_{p>m}\frac{2F_p\xi_p^2}{\l_1-\l_p}}
+2\sum_{p>m}\frac{F_p u_{11p}^2}{\l_1(\l_1-\l_p)}-\sum_{i=1}^n\frac{F_i u_{11i}^2}{\l_1^2}-C.
\end{align*}
Next, by \eqref{eqn.lambda-max-derivative}, it follows that
\begin{gather*}
\xi_1=\xi_2=\cdots=\xi_m=u_{111}=(\l_1)_1=\l_1b_1,\
\\
\l_1b_i=u_{11i}=u_{1i1}=0\
\text{for all}\ i\in\Z_{[2,m]}.
\end{gather*}
Moreover, since
\[
\l_1
\geq A\frb{\max_{B_9\times[-M,M]}f}^{1/(k-l)}
\geq A F(\l)^{1/(k-l)},
\]
Theorem \ref{thm.quotient-concavity}
and $\abs{\sum_{i=1}^nF_i\xi_i}=\abs{\frb{f(x,u)}_1}=\abs{f_{x_1}+f_u u_1}\leq C$ give
\[
-\sum_{i,j=1}^nF_{ij}\xi_i\xi_j
+\sum_{i>m}\frac{2F_i\xi_i^2}{\l_1-\l_i}
\geq(1+\d)\frac{F_1\xi_1^2}{\l_1}
-\frac{K}{F}\frb{\sum_{i=1}^nF_i\xi_i}^2
\geq(1+\d)\frac{F_1\xi_1^2}{\l_1}-C\l_1,
\]
where in the last inequality we used $F\geq\min f>0$ and $\l_1\geq c>0$.
Consequently,
\begin{align*}
F^{ij}b_{ij}
&\geq(1+\d)F_1b_1^2
+2\sum_{i>m}\frac{\l_1F_i}{\l_1-\l_i}b_i^2
-\sum_{i=1}^nF_i b_i^2-C
\\
&\geq\d F_1b_1^2+\sum_{i>m}F_i b_i^2-C
\geq\d F^{ij}b_i b_j-C.
\end{align*}
This proves \eqref{eqn.jacobi} in the viscosity sense.
\end{proof}

\medskip

With the Jacobi inequality established,
we can now complete the proof of the interior Hessian estimate
by following the framework of \cite{LT26}.
For completeness, we record the main estimates,
while the detailed arguments are those of \cite{LT26}.

\begin{proof}[Proof of Theorem \ref{thm.main}]
Fix $\d=\frac14$, and let $A$ be the corresponding constant in
Theorem \ref{thm.quotient-concavity}. Set
\[
\ol\l
:=
\max\set{
1,
A\frb{\max_{B_9\times[-M,M]}f}^{1/(k-l)}
}
\]
and
\[
b:=\log\max\set{\l_{\max}(D^2u),\ol\l}.
\]
By Proposition \ref{thm.jacobi} and the stability of viscosity
subsolutions under taking a maximum,
\begin{equation}\label{eqn.global-jacobi}
F^{ij}b_{ij}
\geq
\d F^{ij}b_i b_j-C
\end{equation}
in the viscosity sense in $B_9$.
We note that, although \cite[Lemmas~3.1, 3.2 and 4.1]{LT26}
are stated for $n\geq3$, their proofs remain valid when $n=2$,
where necessarily $(k,l)=(2,1)$.
We may therefore apply the Legendre transform argument in
\cite[Lemma~3.1]{LT26} to \eqref{eqn.global-jacobi}.
Let $w$ be the Legendre transform of
$u+\frac12\abs{x}^2$, and set
\[
G(D^2w)
:=
-F\frb{-I+(D^2w)^{-1}},
\quad
b^*(y):=b(x(y)).
\]
Then
\[
\sum_iG^{ii}b_{ii}^*\geq-C
\]
in the viscosity sense.
The local maximum principle in \cite[Lemma~3.2]{LT26},
followed by the change of variables from $y$ to $x$, gives
\begin{equation}\label{eqn.mean-value-b}
b(0)
\leq
C\int_{B_1}
b\,\s_{k-1}(D^2u)d x+C.
\end{equation}

We next follow the integration by parts argument in Section~4 of
\cite{LT26}. Define
\[
H^{ij}
:=
\s_lF^{ij}
=
\s_k^{ij}-f\s_l^{ij}.
\]
The estimate in \cite[Lemma~4.1]{LT26} applies almost everywhere
to the function $b$, since $Db=0$ almost everywhere on $\set{\l_{\max}\leq\ol\l}$.
Repeating the argument leading to \cite[(4.4)]{LT26}, we obtain
\begin{equation}\label{eqn.pre-energy-bound}
b(0)
\leq
C\int_{B_3}H^{ij}b_i b_jd x
+
C\int_{B_3}\s_{k-1}(D^2u)d x
+C.
\end{equation}
Multiplying \eqref{eqn.global-jacobi} by $\s_l$ gives
\[
H^{ij}b_{ij}
\geq\d H^{ij}b_i b_j-C\s_l.
\]
If $l=k-1$, then $\s_l=\s_{k-1}$.
If $l=k-2$, the Newton--Maclaurin inequality and
$\s_k=f\s_{k-2}$ imply $\s_{k-2}\leq C\s_{k-1}$.
Thus, in either case,
\begin{equation}\label{eqn.weighted-jacobi}
H^{ij}b_{ij}
\geq
\d H^{ij}b_i b_j-C\s_{k-1}.
\end{equation}
As in the proof of \cite[Theorem~1.2]{LT26},
\eqref{eqn.weighted-jacobi} also holds in the distributional sense.
Testing it with a cutoff supported in $B_4$
and using \cite[Lemma~4.1]{LT26}, we obtain
\begin{equation}\label{eqn.energy-bound}
\int_{B_3}H^{ij}b_i b_jd x
\leq
C\int_{B_4}\s_{k-1}(D^2u)d x.
\end{equation}
Combining \eqref{eqn.pre-energy-bound} and
\eqref{eqn.energy-bound} yields
\begin{equation}\label{eqn.b-sigma-k-1}
b(0)
\leq
C\int_{B_4}\s_{k-1}(D^2u)d x+C.
\end{equation}

Finally, repeating the integration by parts argument at the end
of the proof of \cite[Theorem~1.2]{LT26}, we have
\[
\int_{B_4}\s_{k-1}(D^2u)d x
\leq
C\int_{B_6}\s_1(D^2u)d x+C
\leq C,
\]
where the last inequality follows from the divergence theorem
and the bound for $Du$.
Hence $b(0)\leq C$, and therefore $\l_{\max}(D^2u(0))\leq e^{b(0)}\leq C$.
Since $u$ is convex, it follows that $\|D^2u(0)\|\leq C$.
This completes the proof.
\end{proof}

\section{A Liouville theorem}\label{sec.Liouville}

In this section, we prove Theorem \ref{thm.liouville}.
The quadratic growth assumption is used first to show that the gradient map
is onto and then, after taking the Legendre transform,
to obtain a uniform quadratic upper bound for the dual function.

\begin{proof}[Proof of Theorem \ref{thm.liouville}]
After increasing $c$ if necessary, we may assume that
\begin{equation}\label{eqn.global-quadratic-lower}
u(x)\geq a|x|^2-c
\quad\text{for all}\ x\in\R^n.
\end{equation}
Set $d:=k-l$.
Since $u$ is convex and $\s_k/\s_l=1$, we have
\[
\l(D^2u)\in\G_k:=\set{\l\in\R^n:\s_j(\l)>0,\ 1\leq j\leq k}.
\]
Let
\[
L(A):=\frb{\frac{\s_k(A)}{\s_l(A)}}^{1/d}.
\]
Then $L$ is elliptic and concave on $\G_k$.
For every positive definite matrix $A$, the reciprocal identity gives
\[
L(A^{-1})
=\frb{\frac{\s_{n-k}(A)}{\s_{n-l}(A)}}^{1/d}
=\frb{\frac{\s_{n-l}(A)}{\s_{n-k}(A)}}^{-1/d}.
\]
Since $\frb{\frac{\s_{n-l}}{\s_{n-k}}}^{1/d}$ is positive and concave on the positive cone,
its reciprocal is convex.
Hence $A\mapsto L(A^{-1})$ is locally convex.
The constant-rank theorem of Bian and Guan \cite[Theorem~1.1]{BG09},
applied to $L(D^2u)=1$,
then shows that $D^2u$ has constant rank in $\R^n$.

For any $p\in\R^n$, \eqref{eqn.global-quadratic-lower}
gives $u(x)-p\cdot x\to+\infty$ as $|x|\to+\infty$.
Thus $u-p\cdot x$ attains its minimum, and hence
$p\in Du(\R^n)$.
Therefore, $Du:\R^n\to\R^n$ is onto.
If the rank of $D^2u$ were smaller than $n$,
then every point would be a critical point of $Du$.
By Sard's theorem, $Du(B_j)$ would have measure zero for every
$j\in\Z_{>0}$, contradicting $\R^n=Du(\R^n)=\bigcup_{j=1}^{\infty}Du(B_j)$.
Therefore,
\begin{equation}\label{eqn.strict-convexity-entire}
D^2u>0
\quad\text{in}\ \R^n.
\end{equation}
In particular, $Du$ is strictly monotone and hence one-to-one.
Together with its surjectivity and the inverse function theorem,
$Du$ is a global $C^3$ diffeomorphism of $\R^n$.

By \eqref{eqn.global-quadratic-lower}, $u$ is coercive.
Since $u$ is strictly convex, it has a unique minimum.
After a translation and subtracting a constant,
we may assume that
\[
u(0)=0,
\qquad
Du(0)=0,
\qquad
u\geq0
\quad\text{in}\ \R^n.
\]
The estimate \eqref{eqn.global-quadratic-lower} remains valid,
with possibly different constants $a>0$ and $c\geq0$,
which we continue to denote by $a$ and $c$.

Let $u^*$ be the Legendre transform of $u$,
\[
u^*(y):=x\cdot y-u(x),
\quad
y=Du(x).
\]
Then $u^*\in C^4(\R^n)$ is strictly convex and
\[
Du^*(y)=x,
\quad
D^2u^*(y)=\frb{D^2u(x)}^{-1}.
\]
By the reciprocal identity,
\begin{equation}\label{eqn.dual-quotient}
\frac{\s_{n-l}(D^2u^*)}{\s_{n-k}(D^2u^*)}=1
\quad\text{in}\ \R^n.
\end{equation}
Moreover, \eqref{eqn.global-quadratic-lower} gives
\begin{equation}\label{eqn.dual-quadratic-upper}
0\leq u^*(y)
\leq c+\frac{|y|^2}{4a}
\quad\text{for all}\ y\in\R^n.
\end{equation}

We next obtain a global upper bound for $D^2u^*$.
Fix $R\geq1$ and $x_0\in B_R$.
Let $\rho=1/20$ and define
\[
v_{R,x_0}(z)
:=\frac{u^*(x_0+\rho Rz)}{(\rho R)^2},
\quad\text{for any}\
z\in B_{10}.
\]
Then $v_{R,x_0}$ satisfies \eqref{eqn.dual-quotient}.
Moreover, by \eqref{eqn.dual-quadratic-upper} and convexity,
\[
\|v_{R,x_0}\|_{C^{0,1}(B_9)}\leq C.
\]
If $n-k\geq1$, Theorem \ref{thm.main} gives $|D^2u^*(x_0)|=|D^2v_{R,x_0}(0)|\leq C$.
If $n-k=0$, then $n-l=d\in\set{1,2}$.
For $n-l=1$, \eqref{eqn.dual-quotient} becomes $\t v_{R,x_0}=1$.
Since $v_{R,x_0}$ is convex, $0\leq D^2v_{R,x_0}\leq I$.
For $n-l=2$, the same estimate follows from the interior Hessian estimate in \cite{M21}.
Since $R\geq1$ and $x_0\in B_R$ are arbitrary, we obtain
\begin{equation}\label{eqn.global-dual-C2}
\|D^2u^*\|_{L^\infty(\R^n)}\leq C.
\end{equation}

It remains to verify the uniform ellipticity of
\eqref{eqn.dual-quotient}.
Let $\mu=(\mu_i):=\l(D^2u^*)$ with $\mu_1\geq\mu_2\geq\cdots\geq\mu_n>0$.
By \eqref{eqn.global-dual-C2}, $\mu_1\leq C$.
If $d=1$, then $q_{n-l}(\mu)=1$.
If $d=2$, then
\[
q_{n-l}(\mu)q_{n-l-1}(\mu)=\frac{\s_{n-l}(\mu)}{\s_{n-l-2}(\mu)}=1.
\]
By Lemma \ref{lem.lambda-k-bounded},
it follows that $q_{n-l-1}(\mu)\leq(l+2)\mu_{n-l-1}\leq C$,
and hence $q_{n-l}(\mu)\geq C^{-1}$.
Therefore, in either case,
\[
\mu_{n-l}
\geq\frac{q_{n-l}(\mu)}{l+1}\geq C^{-1}.
\]
Together with \eqref{eqn.global-dual-C2}, this shows that
$\l(D^2u^*)$ remains in a compact subset of $\G_{n-l}$.
Consequently, $\frb{\frac{\s_{n-l}}{\s_{n-k}}}^{1/d}$ is uniformly elliptic along $u^*$.

For $R\geq1$, set $v_R(z):=\frac{u^*(Rz)}{R^2}$ for any $z\in B_2$.
Then $v_R$ satisfies \eqref{eqn.dual-quotient}.
By \eqref{eqn.dual-quadratic-upper},
\eqref{eqn.global-dual-C2}, and the uniform ellipticity established above,
the Evans--Krylov estimate \cite{E82,K82} gives
some $\a\in(0,1)$ such that
\[
R^\a[D^2u^*]_{C^\a(B_R)}
=[D^2v_R]_{C^\a(B_1)}
\leq C.
\]
For any $y,z\in\R^n$, choosing $R$ sufficiently large so that
$y,z\in B_R$ and then letting $R\to+\infty$, we obtain
\[
|D^2u^*(y)-D^2u^*(z)|
\leq CR^{-\a}|y-z|^\a\to0.
\]
Hence $D^2u^*$ is constant in $\R^n$,
and thus $u^*$ is a quadratic polynomial, and so is $u$.
\end{proof}

\subsection*{Acknowledgments}

This work was partially supported by NSFC 12171389 and NSFC 11801015.

\renewcommand\refname{References}

\end{document}